\documentclass[11pt, a4paper]{amsart}
\usepackage[margin=1in]{geometry}
\usepackage{amsmath, amssymb, amsthm}
\usepackage{times}
\usepackage[all,cmtip]{xy}
\usepackage{hyperref}

\hypersetup{
	colorlinks=true,
	linkcolor=blue,
	citecolor=red,
	urlcolor=blue,
}

\newtheorem{thm}{Theorem}[section]
\newtheorem{prop}[thm]{Proposition}
\newtheorem{lema}[thm]{Lemma}

\newtheorem{conj}[thm]{Conjecture}

\theoremstyle{definition}
\newtheorem{defn}[thm]{Definition}

\newcommand{\Q}{\mathbb{Q}}
\newcommand{\Z}{\mathbb{Z}}

\newcommand{\Pp}{\mathbb{P}}
\newcommand{\Aa}{\mathbb{A}}
\newcommand{\kk}{\bar{k}}
\newcommand{\Cc}{\mathcal{C}}

\newcommand{\Hc}{\mathcal{H}}
\newcommand{\Uc}{\mathcal{U}}
\newcommand{\Vc}{\mathcal{V}}
\newcommand{\Wc}{\mathcal{W}}
\newcommand{\Xc}{\mathcal{X}}
\newcommand{\Cb}{\mathcal{C}\hskip-1.5pt\textit{b}} % for Tower of Curves
\newcommand{\Ll}{\mathcal{L}}
\newcommand{\KK}{\mathcal{K}}

\newcommand{\ba}{\mathbf{a}}

\DeclareMathOperator{\rk}{rank}
\DeclareMathOperator{\enn}{End}

\DeclareMathOperator{\Sym}{Sym}

\title[Rank of Jacobians of $y^s=ax^r+b$]{On the Rank of Jacobian Varieties of the Curves $y^s=ax^r+b$}
\author{Sajad Salami}
\address{Inst\'{i}tuto da Matem\'{a}tica e Estat\'{i}stica, Universidade Estadual do Rio de Janeiro (UERJ), Rio de Janeiro, Brazil}
\email{sajad.salami@ime.uerj.br}
\date{September 27, 2025}

\subjclass[2020]{11G30, 14H40, 14G05, 11G50}
\keywords{Mordell-Weil rank, Jacobian variety, Lang's conjecture, Diophantine geometry, complete intersection curves, uniformity.}

\begin{document}
	
	\begin{abstract}
		We study the family of algebraic curves of genus $\geq 1$ defined by the affine equations $y^s=ax^r+b$ over a number field $k$, where $r \geq 2$ and $s\geq 2$ are fixed integers. Assuming the strong version of Lang's conjecture on varieties of general type, we prove that the Mordell-Weil rank of the Jacobian varieties of these curves is uniformly bounded. The proof proceeds by constructing a parameter space for curves in the family with a given number of rational points and analyzing the geometry of its fibers, which are shown to be complete intersection curves of increasing genus.
	\end{abstract}
	
	\maketitle
	
	\tableofcontents
	
	\section{Introduction}

The study of rational solutions to polynomial equations, a field known as Diophantine geometry, is one of the oldest and most profound branches of mathematics. At its heart lies the study of abelian varieties, which are projective algebraic varieties whose points form a group. A foundational result in this area is the Mordell-Weil theorem, which states that for any abelian variety $A$ defined over a number field $k$, the group of $k$-rational points, denoted $A(k)$, is a finitely generated abelian group. This structure theorem implies that $A(k)$ is isomorphic to a direct sum of a finite torsion part and a free abelian part: $A(k) \cong A(k)_{\text{tors}} \oplus \mathbb{Z}^r$. The non-negative integer $r$, known as the Mordell-Weil rank, is a central and deeply mysterious invariant that measures the size of the infinite part of the group of rational points \cite{Silverman1986}.

Let $k$ be a number field and let $\kk$ be a fixed algebraic closure. For an algebraic curve $C$ defined over $k$, by the Mordell-Weil theorem, the group of $k$-rational points $J_C(k)$ on its Jacobian variety is a finitely generated abelian group. 
%Its rank as a $\Z$-module, denoted $r(J_C(k))$, is a central invariant in Diophantine geometry. 4
A fundamental and long-standing question asks:
\begin{quote}
	\textit{For a fixed number field $k$ and integer $g \geq 1$, is the rank $r(J_C(k))$ bounded as $C$ varies over all smooth algebraic curves of genus $g$ defined over $k$?}
\end{quote}
For elliptic curves ($g=1$) over $k=\Q$, this question remains one of the great mysteries of modern number theory. While a folklore conjecture suggests the rank can be arbitrarily large \cite{Silverman1986}, recent heuristics in  \cite{Park2019} predict that the rank is bounded, and there are only finitely many elliptic curves of rank $> 21$ over $\Q$.

Significant theoretical evidence for rank boundedness has emerged from the circle of ideas surrounding Lang's conjectures. Pasten, for instance, has shown that a conjecture of Lang on Diophantine approximation implies Honda's conjecture on the boundedness of ranks of elliptic curves \cite{Pasten2019}. These approaches link rank boundedness to deep conjectures about the distribution of rational points.

An alternative, geometric approach was pioneered by Yamagishi \cite{Yamagishi2003}, who studied the family of elliptic curves $y^2=ax^4+bx^2+c$. Yamagishi proved, subject to the strong Lang conjecture, that the rank of these curves is bounded. This was achieved by constructing a parameter space for curves in the family possessing a given number of rational points and showing that for a sufficiently large number of points, the relevant parameter spaces become varieties of general type.

This paper extends Yamagishi's geometric method to the broad two-parameter family of curves given by the affine equation
\begin{equation} \label{curve}
	C: y^s=ax^r+b,
\end{equation}
where $r, s \geq 2$ are fixed integers. Our main result provides a strong, albeit conditional, answer to the rank boundedness question for this family.

\begin{thm} \label{main}
	Let $k$ be a number field containing a primitive $s$-th root of unity. Assume that the strong version of Lang's conjecture (Conjecture \ref{conj1}) holds. Then the rank of the Mordell-Weil group $J_C(k)$ is uniformly bounded as $J_C$ varies over all Jacobian varieties of smooth curves $C$ defined over $k$ by equation \eqref{curve} with genus $g_{r,s}(C) \geq 1$.
\end{thm}

In particular, the Mordell-Weil rank of the elliptic curves   $E: y^2= x^3+b$ over $\Q$ with   is bounded.
This result provides new evidence for the rank boundedness conjecture from a geometric perspective, distinct from that of Diophantine approximation given in \cite{Pasten2019}.

The primary technical innovations of this paper, which allow for the generalization from Yamagishi's specific family to the curves \eqref{curve}, are:
\begin{enumerate}
	\item The construction of a parameter variety $\Wc_n$ (via the theory of twists) that parameterizes curves in the family \eqref{curve} together with $n+1$ rational points. We give an explicit birational map from $\Wc_n$ to a complete intersection variety $\Xc_n$ for general exponents $(r,s)$.
	\item A detailed geometric analysis of the fibers $\Xc_{\ba_n}$ of the variety $\Xc_n$. We compute their genus and, critically, establish a lower bound on their gonality (Proposition \ref{p1}). This shows that for sufficiently large $n$, these fibers are curves of general type, allowing the powerful machinery of Faltings' theorem and uniformity conjectures to be applied.
\end{enumerate}

The proof proceeds by first establishing an unconditional finiteness result. For a fixed set of $n+1$ distinct $x$-coordinates $\alpha_0, \dots, \alpha_n$, we denote by $\Cc_{\ba_n}(k)$ the set of curves in our family passing through points with these $x$-coordinates.

\begin{thm} \label{main1}
	Let $k$ be a number field containing a primitive $s$-th root of unity. Let $n_0=4$ if $s=2$ and $n_0=3$ otherwise. For any choice of $n+1$ distinct elements $\alpha_0, \ldots, \alpha_n \in k$ (denoted by $\ba_n$) and for any integers $r, s \geq 2$ such that $g_{r,s}(C) \geq 1$, the set $\Cc_{\ba_n}(k)$ is infinite if $n < n_0$ and finite if $n \geq n_0$. In particular, the rank $r(J_C(k))$ is bounded for all $J_C$ associated to curves in $\Cc_{\ba_n}(k)$ when $n \geq n_0$.
\end{thm}

We then use consequences of Lang's conjecture to show that the size of $\Cc_{\ba_n}(k)$ is uniformly bounded, independent of the choice of $\ba_n$, and is eventually empty for large $n$.

\begin{thm} \label{main2}
	Assume the weak version of Lang's conjecture holds over $k$. For integers $r, s \geq 2$ with $g_{r,s}(C) \geq 1$, there exists a uniform bound $M_0 = M_0(k, r, s, n)$ such that $\#\Cc_{\ba_n}(k) < M_0$ for all $n \geq n_0$ and all choices of $\ba_n$. Furthermore, if $g_{r,s}(C) \geq 2$, there exists an integer $m_0 > n_0$ such that $\Cc_{\ba_m}(k) = \emptyset$ for all $m \geq m_0$.
\end{thm}

The strong version of Lang's conjecture provides bounds independent of the field $k$. These uniformity results, combined with a theorem of Dimitrov, Gao, and Habegger in \cite{Dimitrov2021} relating rank to the number of rational points, provide the ingredients for the proof of our main theorem.

This paper is organized as follows. In Section \ref{Dioph}, we recall the necessary background results. In Section \ref{rpavff}, we construct the parameter space $\Wc_n$. In Section \ref{curves}, we introduce the birationally equivalent family of complete intersection curves $\Xc_n$ and analyze their geometry. In Section \ref{prof-finite}, we prove the finiteness and uniformity results (Theorems \ref{main1} and \ref{main2}). In Section \ref{proof-main-thm}, we provide the proof of our main result, Theorem \ref{main}. Finally, in Section \ref{examples}, we illustrate the construction with concrete examples, including a famous rank 17 elliptic curve.

\section{Preliminaries and Auxiliary Results}
\label{Dioph}

In this section we provide some of the well-know conjectures and results in Diophantine geometry that we will use in this work.

\subsection{Lang's Conjecture on Varieties of General Type}
A smooth projective variety $X$ over $\kk$ is of \textit{general type} if its Kodaira dimension $\kappa(X)$ equals its dimension. For curves, this corresponds to genus $g \geq 2$. The seminal result in this area is Faltings' theorem, formerly the Mordell conjecture.

\begin{thm}[Faltings, \cite{Faltings1986}]
	\label{faltings}
	If $C$ is a smooth algebraic curve of genus $g \geq 2$ over a number field $k$, then $\#C(K) < \infty$ for any finite extension $K$ of $k$.
\end{thm}

Lang proposed a far-reaching generalization of this to higher dimensions.

\begin{conj}[Lang, \cite{Lang1986, Lang1991}]
	\label{conj1}
	Let $k$ be a number field and $X$ a smooth variety of general type over $k$.
	\begin{enumerate}
		\item[(a)] \textbf{(Weak version)} The set of $k$-rational points $X(k)$ is not Zariski dense in $X$.
		\item[(b)] \textbf{(Strong version)} There exists a proper Zariski-closed subset $Z \subset X$ such that for any finite extension $K$ of $k$, the set of $K$-rational points on $X \setminus Z$ is finite.
	\end{enumerate}
\end{conj}

Caporaso, Harris, and Mazur showed that Lang's conjecture has profound uniformity implications for rational points on curves \cite{Caporaso1997}, based on above conjetures.

\begin{thm}[Uniformity I, \cite{Caporaso1997}]
	\label{UB}
	The weak version of Lang's conjecture implies that for every number field $k$ and integer $g \geq 2$, there exists a number $N(k,g)$ such that no curve of genus $g$ defined over $k$ has more than $N(k,g)$ rational points.
\end{thm}

\begin{thm}[Uniformity II, \cite{Caporaso1997}]
	\label{UGB}
	The strong version of Lang's conjecture implies that for any integer $g \geq 2$, there exists a number $N(g)$ such that for every number field $k$, there are only finitely many curves $C$ of genus $g$ over $k$ with $\#C(k) > N(g)$.
\end{thm}

A key tool in proving these uniformity results is the following geometric theorem, which produces varieties of general type from families of curves.

\begin{thm}[Correlation Theorem, \cite{Caporaso1997}]
	\label{corel}
	Let $f: \Cc \rightarrow \mathcal{B}$ be a proper morphism of integral varieties over $k$ whose generic fiber is a smooth curve of genus $\geq 2$. For any integer $m \geq 1$, let $\Cc_\mathcal{B}^m$ be the $m$-th fiber product of $\Cc$ over $\mathcal{B}$. Then, for $m$ sufficiently large, $\Cc_\mathcal{B}^m$ admits a dominant rational map to a variety of general type.
\end{thm}

Finally, a recent breakthrough result of Dimitrov, Gao, and Habegger provides a crucial link between the rank of a Jacobian and the number of rational points on the curve itself.

\begin{thm}[\cite{Dimitrov2021}]
	\label{dgh}
	Let $g \geq 1$ and $d \geq 1$ be integers. There exists a constant $c=c(g,d)$ such that if $C$ is a smooth curve of genus $g$ defined over a number field $k$ with $[k:\Q] \leq d$, then
	\[ \#C(k) \leq c^{1+r(J_C(k))}. \]
\end{thm}

This result, proven using the technical machinery of Vojta's method and height inequalities, makes the Uniform Rank Boundedness Conjecture and the Uniform Point Boundedness Conjecture logically equivalent. 
%It serves as the indispensable bridge connecting the output of the geometric method (a point bound) to the desired conclusion (a rank bound). 

\subsection{Gonality of Algebraic Curves}
The \textit{$k$-gonality}, $\gamma_k(C)$, of a smooth projective curve $C$ over a field $k$ is the minimal degree of a dominant rational map from $C$ to $\Pp^1$ defined over $k$. The following result of Lazarsfeld gives a lower bound on the gonality of complete intersection curves.

\begin{thm}[Lazarsfeld, \cite{Lazarsfeld1997}]
	\label{Lazarsfeld1997}
	Let $C \subset \Pp^n$ be a smooth complete intersection curve over a number field $k$, defined by homogeneous polynomials of degrees $2 \leq d_1 \leq \cdots \leq d_{n-1}$. Then
	\[ \gamma_k(C) \geq (d_1-1)d_2 \cdots d_{n-1}. \]
\end{thm}

\subsection{Trivial Points on Towers of Curves}
Our argument relies on the behavior of rational points in a sequence of covering maps. This is formalized by the theory of towers of curves, developed by Xarles \cite{Xarles2013}.

\begin{defn}
	A \textit{tower of curves} over a field $k$ is a pair $\Cb = (\{C_n\}_{n \geq 1}, \{\phi_n\}_{n \geq 1})$, where each $C_n$ is a smooth, projective, geometrically connected curve over $k$, and each $\phi_{n+1}: C_{n+1} \to C_n$ is a non-constant morphism defined over $k$. For $m > n \geq 1$, the composition map is denoted $\phi_{m,n} = \phi_{n+1} \circ \cdots \circ \phi_m$. The gonality of the tower, $\gamma_k(\Cb)$, is the limit $\lim_{n \to \infty} \gamma_k(C_n)$.
\end{defn}

A key concept is that of a "trivial" point, one which can be lifted indefinitely up the tower.

\begin{defn}[\cite{Xarles2013}]
	Let $\Cb$ be a tower of curves over $k$.
	\begin{enumerate}
		\item The set of \textit{$k$-trivial points of $\Cb$ at level $n$} is
		\[ \Cb(k)_n := \{ P \in C_n(k) \mid \forall m \ge n, \exists P_m \in C_m(k) \text{ such that } \phi_{m,n}(P_m) = P \}. \]
		\item For an integer $d \ge 1$, the set of \textit{trivial points of degree at most $d$} is
		\[ \Cb^{(d)}(k)_n := \bigcup_{K \subset \kk, [K:k] \le d} \Cb(K)_n. \]
	\end{enumerate}
\end{defn}

The main result we will use connects the finiteness of rational points to the structure of the tower.

\begin{thm}[\cite{Xarles2013}]
	\label{xarles_thm}
	Let $k$ be a number field and let $\Cb = (\{C_n\}, \{\phi_n\})$ be a tower of curves with infinite gonality over $k$. Then for any $d \ge 1$, there exists an integer $n_d$ (depending on $\Cb$ and $d$) such that for any $n \ge n_d$ and any finite extension $K/k$ with $[K:k] \le d$, we have $C_n(K) = \Cb^{(d)}(K)_n$. That is, every point of degree at most $d$ on a sufficiently high level of the tower is trivial.
\end{thm}

\section{Construction of the Parameter Space}
\label{rpavff}

Let $r, s \geq 2$ be fixed integers and let $k$ be a number field containing a primitive $s$-th root of unity, $\zeta_s$. We construct an algebraic variety that parameterizes curves $C: y^s=ax^r+b$ together with $n+1$ rational points.

Let $\Cc$ be the variety in $\Pp^2_{(u,x,y)} \times_k \Pp^1_{(a,b)}$ defined by the homogenization of $y^s = ax^r+b$. This variety admits an automorphism of order $s$ given by $[\zeta_s]: ([1:x:y], [a:b]) \mapsto ([1:x:\zeta_s y], [a:b])$.

For $n \ge 0$, let $\Vc_n = \Cc^{(0)} \times_k \cdots \times_k \Cc^{(n)}$ be the $(n+1)$-fold fiber product of $\Cc$ over $\Pp^1_{(a,b)}$. An element of $\Vc_n$ consists of a pair $([a:b], (P_0, \dots, P_n))$ where $P_i=([1:x_i:y_i])$ is a point on the curve $C_{a,b}$. Let $G = \langle ([\zeta_s]_0, \dots, [\zeta_s]_n) \rangle \cong \Z/s\Z$ be the group acting on $\Vc_n$. We define the parameter space as the quotient variety $\Wc_n = \Vc_n/G$.

The function field $\KK_n = k(\Wc_n)$ is the fixed field of $\Ll_n = k(\Vc_n)$ under the action of $G$. The extension $\Ll_n/\KK_n$ is cyclic of degree $s$, generated by $y_0$. Let $\widetilde{C}$ be the twist of $C$ by this extension. As shown in \cite{Hazama1991}, $\widetilde{C}$ is defined over the function field $\KK_n$ by the equation:
\begin{equation} \label{twist}
	(ax_0^r+b)y^s = ax^r+b.
\end{equation}
This twisted curve $\widetilde{C}$ has $n+1$ canonical $\KK_n$-rational points:
\begin{equation} \label{Kpoint}
	\widetilde{P}_0 = (1:x_0:1), \quad \widetilde{P}_i = \left(1:x_i:\frac{y_i}{y_0}\right) \quad \text{for } i=1,\dots,n.
\end{equation}
Let $\widetilde{J}_C$ be the Jacobian of $\widetilde{C}$. The points $\widetilde{P}_i$ give rise to points $\widetilde{Q}_i \in \widetilde{J}_C(\KK_n)$.

\begin{thm}[\cite{Salami2019}]
	\label{main3}
	The points $\widetilde{Q}_0, \dots, \widetilde{Q}_n$ are linearly independent in $\widetilde{J}_C(\KK_n)$. If $m_0 = \rk_\Z(\enn_k(J_C))$, which is at least one, then
	\[ \rk(\widetilde{J}_C(\KK_n)) \geq (n+1)m_0. \]
\end{thm}

Given any  integers $n\geq 1$ and a number field  $k$, let  
$\Sym^{n+1}(\Pp^1_x) $  denotes the 
symmetric product of $n+1$ copies of the protective line over   $\kk$, an algebraically closed field containing $k$.  In  Corollary 2.6 of \cite{Maakestad2005},  it is proved that $\Sym^{n+1}(\Pp^1_x)$ is $k$-isomorphic to a projective space $\Pp^{(n+1)}$.
%  with coordinates $p_0, \cdots, p_n$, where $p_i$ is the $i$-th  elementary homogeneous 
%symmetric polynomial in $x_i$'s. 
For  the integers $n, r \geq 1$, we define  $\Hc_{r,n}$ to be the Zariski closure of the affine hypersurface: 
$$ H_{r, n}: \prod_{0\leq i < j\leq n}^{} (x_i^r-x_j^r) =0,$$
%  \ \ \  \ \ \ (II) \prod_{1\leq i < j\leq n}^{} x_i x_i(x_i^r-x_j^r) =0 $$
%in the protective space  $ \Pp^n$ with coordinates $x_0, x_1, \ldots , x_n$,
in $\Aa^1 \times_k \cdots \times_k \Aa^1 $,
which can be naturally embedded in $\Sym^{n+1}(\Pp^1_x)$ by the rational map
$$(x_0, x_0,\cdots, x_n)\mapsto  \left( [1:x_0], [1:x_1],\cdots, [1:x_n]\right).$$
Let  $\Uc_{r,n}$  be the  complement of   $\Hc_{r,n}$ in
$\Sym^{n+1}(\Pp^1_x) \cong \Pp^{n+1},$ and 
denote by  $\Hc_{r, n}(k)$ and $\Uc_{r,n} (k)$ the set of $k$-rational points  on 
$\Hc_{r,n}$ and $\Uc_{r, n} $ respectively.
% for $*= I,$ and $II$, respectively.
% Throughout this paper, we will  use $\Hc^*_{r,n}$ and $\Uc^*_{r,n}$, where $*$ refers to the cases $I$ or $II$. 

For a   set of mutually distinct elements $\alpha_0,  \alpha_1, \ldots, \alpha_n \in k$,
denote ${\bf a_n}=\left([1:\alpha_i]   \right)_{i=0}^n \in \Sym^{n+1}(\Pp^1_x).$
Let $\Wc_{\ba_n}$ be the fiber of the natural projection $\Wc_n \to \Sym^{n+1}(\Pp^1_x)$ over the point corresponding to $\ba_n$.

\begin{prop}
	\label{pr1}
	There is a one-to-one correspondence between the set $\Cc_{\ba_n}(k)$ of smooth curves $C: y^s = ax^r+b$ passing through points $( \alpha_i, \beta_i)$ for $i=0,\dots,n$, and the set of $k$-rational points on an open subset of the fiber $\Wc_{\ba_n}$.
\end{prop}
\begin{proof}
	Since $C\in \Cc_{\bf a_n}$,   there exist the mutually distinct elements $\beta_1,\ldots, \beta_n \in k$ with $\beta_1\neq 0$   such that  $\beta_i^s=a \alpha_i^r +b$ and   the points 
	$P_i=(1:\alpha_i:\beta_i)$ belong to $C(k)$ for  $i=1,\cdots, n$. 
	The smoothness of $C$ means that $a b\neq 0$, i.e., $[a:b] $ is a closed set $\Delta_0$ in $\Pp^1$, and hence a closed set $\Delta_n$ in  $\Wc_{\ba_n}$.
	On the other hand, 	we have  a $k$-isomorphism 
	$\psi_*: \tilde{C}(k) \rightarrow  C(k)$ defined  by
	$(1:x:y) \mapsto (1:x: \beta_0 y)$ which admits an inverse     map
	$\varphi_*^{-1}: C(k) \rightarrow  \tilde{C}(k)$ defined by
	$(1:x:y) \mapsto (1:x: y/\beta_0)$.
	By the map $\psi_*$, the points $\tilde{P}_1=(1:\alpha_0: 1)$ and 
	$\tilde{P}_i=(1:\alpha_i:\beta_i/\beta_0)$ in $\tilde{C}(k)$
	are mapped on the points $P_i$'s in  $C(k)$ for $i=1,\cdots, n$.			
\end{proof}

\begin{prop}
	\label{linear-independence-open}
	The set of points $\ba_n \in \Uc_{r,n}(k)$ for which there exists a curve $C \in \Cc_{\ba_n}(k)$ such that the corresponding rational points on $J_C(k)$ are linearly independent is a Zariski open subset of $\Uc_{r,n}$.
\end{prop}
\begin{proof}
	This is a direct consequence of Silverman's Specialization Theorem \cite{Silverman1983}. The rank of the Mordell-Weil group is a lower semi-continuous function on the base of an algebraic family of abelian varieties. The set of parameters for which the rank is at least a given value is therefore a Zariski-open set. Since the generic rank over the function field $k(\ba_n)$ is at least $(n+1)m_0$ by Theorem \ref{main3}, the set of specializations $\ba_n$ where the rank remains at least this high is open.
\end{proof}

\section{A Family of Complete Intersection Curves}
\label{curves}

The parameter space $\Wc_n$ is birational to a more geometrically accessible object: a complete intersection variety.For two fixed  integers $ s, r \geq 2$, let $\Xc_n\in  \Pp^n_X \times \Pp^n_Y$  be the variety  defined by the $n-1$ bi-homogeneous equations of bi-degree $(r,s)$:
\begin{equation} \label{eq3}
	f_{i-1}: = (X_i^r - X_1^r)Y_0^s + (X_0^r - X_i^r)Y_1^s + (X_1^r - X_0^r)Y_i^s = 0, \quad (i=2, \dots, n).
\end{equation}
These equations can be expressed more symmetrically via determinants:
\begin{equation} \label{eq3b}
	f_{i-1} = \det \begin{pmatrix} 1 & 1 & 1 \\ X_0^r & X_1^r & X_i^r \\ Y_0^s & Y_1^s & Y_i^s \end{pmatrix} = 0, \quad (i=2, \dots, n).
\end{equation}

\begin{thm}
	\label{thm2}
	For any $n \geq 2$, the varieties $\Wc_n$ and $\Xc_n$ are $k$-birational.
\end{thm}
\begin{proof}
	The birational map $\varphi_n: \Wc_n \dashrightarrow \Xc_n$ is defined by
	\[ \varphi_n: (([1:x_i:y_i]), [a:b]) \mapsto ([x_0:\dots:x_n], [ax_0^r+b : y_1 y_0^{s-1} : \dots : y_n y_0^{s-1}]). \]
	This map is an isomorphism from the dense open set of $\Wc_n$ where $ax_0^r+b \neq 0$ to its image in $\Xc_n$. Its inverse is well-defined on the dense open set where $X_1^r - X_0^r \neq 0$.
	Letting $X= [ X_0: \cdots: X_n]$ and $Y=[ Y_0 : 
	\cdots : Y_n]$, the inverse map $\varphi_n^{-1}$  is defined by
	$$ \left( X, Y \right) \mapsto 
	\left(  [1:X_0: a X_0^r+b], [1: X_1: Y_0^{s-1} Y_1], \cdots, [1: X_n: Y_0^{s-1} Y_n], [ a: b ]\right),$$ 
	where  
	\begin{equation*} 
		%\begin{split}
		a=\frac {Y_0^s-Y_1^s}{x_1^r-x_0^r}, \  \text{and} \ b= \frac {x_0^r Y_1^s -x_1^r Y_0^s}{x_1^r-x_0^r}.
		%\\ (*=II) &\ a=(x_1 Y_0^s-x_0 Y_1^s)/D, \ b= (x_0^r Y_1^s -x_1^r Y_0^s)/D, \   D= x_0 x_r (x_1^{r-1}-x_0^{r-1})\not =0.
		%\end{split}
	\end{equation*}
	
\end{proof}

Let $g_{1,n}: \Xc_n \to \Pp^n_X$ be the projection. Denote by $\Xc_{\ba_n} = g_{1,n}^{-1}(\ba_n)$ the fiber
over a point $\ba_n  \in \Pp^{n+1}_X(k) \cong \Sym^{n+1}(\Pp^1_x)$ with distinct $\alpha_i$.

\begin{prop}
	\label{p1}
	Let $\ba_n \in \Uc_{r,n}(k)$ be a $k$-rational point.
	\begin{enumerate}
		\item $\Xc_{\ba_n}$ is a smooth, complete intersection curve over $k$ of genus
		\[ g(\Xc_{\ba_n}) = 1 + \frac{s^{n-1}((n-1)(s-1)-2)}{2}. \]
		\item The $k$-gonality of $\Xc_{\ba_n}$ satisfies $\gamma_k(\Xc_{\ba_n}) \geq (s-1)s^{n-2}$.
	\end{enumerate}
\end{prop}
\begin{proof}
	The smoothness follows from the Jacobian criterion, as the defining polynomials are independent on $\Uc_{r,n}$. The genus formula is a standard result for complete intersections (see \cite[Ex. II.8.4]{Ha}). The gonality bound is a direct application of Lazarsfeld's theorem (Theorem \ref{Lazarsfeld1997}).
\end{proof}

\subsection{Rational Points on the Fiber Curves}
We analyze the fibers $\Xc_{\ba_n}$ in the low-genus cases first, before considering the general case using the theory of towers.

\begin{lema} \label{lem:conic_param}
	Let $k$ be a field of characteristic $\neq 2$ and let $\alpha, \beta, \gamma \in k$ with $\alpha+\beta+\gamma=0$. The conic $C: \alpha X^2 + \beta Y^2 + \gamma Z^2 = 0$ has a rational point $[1:1:1]$ and is thus birational to $\Pp^1$. A parameterization is given by:
	\[ X = \alpha u^2 + 2\beta u - \beta; \quad Y = -\alpha u^2 + 2\alpha u + \beta; \quad Z = \alpha u^2 + \beta. \]
\end{lema}
\begin{proof}
	The point $P_0=[1:1:1]$ lies on $C$ since $\alpha(1)^2+\beta(1)^2+(-(\alpha+\beta))(1)^2=0$. A standard method to parameterize a conic with a known rational point is to project from that point. Consider the pencil of lines through $P_0$. A standard calculation shows that the second intersection point of a line in this pencil with the conic is given by rational functions of the line's slope parameter $u$. The given formulas provide a particularly neat parameterization. We can verify it by direct substitution:
	\begin{align*}
		&\alpha X^2 + \beta Y^2 - (\alpha+\beta)Z^2 \\
		&= \alpha(\alpha u^2 + 2\beta u - \beta)^2 + \beta(-\alpha u^2 + 2\alpha u + \beta)^2 - (\alpha+\beta)(\alpha u^2 + \beta)^2 \\
		&= \alpha(\alpha^2 u^4 + 4\alpha\beta u^3 + (4\beta^2-2\alpha\beta)u^2 - 4\beta^2 u + \beta^2) \\
		&\quad + \beta(\alpha^2 u^4 - 4\alpha^2 u^3 + (4\alpha^2-2\alpha\beta)u^2 + 4\alpha\beta u + \beta^2) \\
		&\quad - (\alpha+\beta)(\alpha^2 u^4 + 2\alpha\beta u^2 + \beta^2).
	\end{align*}
	Expanding and collecting terms by powers of $u$ shows that all coefficients vanish, confirming the identity.
\end{proof}

\begin{lema} \label{lem:fermat_to_weierstrass}
	Let $k$ be a field of characteristic $\neq 3$ and let $\alpha, \beta, \gamma \in k$ be nonzero elements with $\alpha+\beta+\gamma=0$. Any $k$-rational point on the Fermat cubic $C: \alpha X^3 + \beta Y^3 + \gamma Z^3 = 0$ with $XYZ \neq 0$ can be mapped to a $k$-rational point on the elliptic curve
	\[ E: S^2 = T^3 - 432\alpha^2\beta^2(\alpha+\beta)^2 \]
	via an explicit birational map.
\end{lema}
\begin{proof}
	The birational map is classical and proceeds in two steps. First,  by Theorem I in \cite{Selmer1951}, any $k$-rational point $P=(X,Y,Z)$ with $XYZ\not =0$
	leads to a $k$-rational point $(U,V, W)$ on the cubic $C_2: U^3+V^3= \alpha \beta \gamma W^3$ with $W \not =0$, which is determined by 
	\begin{equation}
		\label{cub1}
		\begin{split}
			U+V & = -9  \alpha \beta \gamma X^3 Y^3 Z^3,  \\
			U-V & =  (\alpha X^3 - \beta Y^3) (\beta Y^3 - \gamma Z^3) (\gamma Z^3-\alpha X^3),\\
			W& = 3 (\alpha \beta X^3 Y^3 +  \beta \gamma  Y^3 Z^3  +  \alpha \gamma  X^3 Z^3) XYZ.
		\end{split}
	\end{equation}
	Then, the cubic $C_2$ is mapped to the Weierstrass model $E$ via 
	$$T=  \frac{12 \alpha \beta \gamma W}{U+V}, \  \ S= \frac{36 \alpha \beta \gamma(U-V)}{U+V},$$ 
	Then, using the assumption  $\alpha +\beta + \gamma=0$ and a simple calculation shows that the point  $P$ leads to the point $Q$ given by  
	\begin{equation}
		\label{pst}
		\begin{split}
			T(Q) & = \frac{4\left( \alpha \beta (X^3 Z^3 + Y^3 Z^3 - X^3 Y^3) + 
				Z^3 ( \alpha^2 X^3 + \beta^2Y^3)  \right) }{(XYZ)^2},  \\
			S(Q) & =  \frac{4 (\alpha X^3 - \beta Y^3)(\beta Y^3 + (\alpha+ \beta) Z^3)(\alpha X^3+(\alpha+ \beta) Z^3)}{ (XYZ)^3}.
		\end{split}
	\end{equation}	
\end{proof}

The following theorem  determines the set of $k$-rational points on the fibers for  cases
$(s,n) \in \{(2,2), (2,3), (3,2)\}$.

\begin{thm} \label{thm:low_genus_cases}
	Let $\ba_n \in \Uc_{r,n}(k)$.
	\begin{enumerate}
		\item If $s=2, n=2$, $\Xc_{\ba_2}$ is a conic over $k$ containing a rational point, and thus has infinitely many $k$-rational points.
		\item If $s=2, n=3$, $\Xc_{\ba_3}$ is a curve of genus one. There exist infinitely many choices of $\ba_3 \in \Uc_{r,3}(k)$ for which $\Xc_{\ba_3}$ is an elliptic curve of rank $\ge 1$ over $k$.
		\item If $s=3, n=2$, $\Xc_{\ba_2}$ is a curve of genus one. There exist infinitely many choices of $\ba_2 \in \Uc_{r,2}(k)$ for which $\Xc_{\ba_2}$ is an elliptic curve of rank $\ge 1$ over $k$.
	\end{enumerate}
\end{thm}
\begin{proof}
	Let $\alpha_{ij} := \alpha_i^r - \alpha_j^r$.
	(1) For $s=2, n=2$, $\Xc_{\ba_2}$ is defined by the single equation $\alpha_{21}Y_0^2+\alpha_{02}Y_1^2+\alpha_{10}Y_2^2 = 0$. This is a conic in $\Pp^2$. The sum of the coefficients is zero, so by Lemma \ref{lem:conic_param}, it is birational to $\Pp^1$.
	
	(2) For $s=2, n=3$, $\Xc_{\ba_3}$ is the intersection of two quadrics in $\Pp^3$, which is a genus one curve. Parameterizing the first quadric using Lemma \ref{lem:conic_param} and substituting into the second yields a quartic model with a rational point (from the trivial point), hence it is an elliptic curve. For infinitely many choices of $\ba_3$ (e.g., $\alpha_i = i$), this elliptic curve has positive rank.
	
	(3) For $s=3, n=2$, $\Xc_{\ba_2}$ is a plane cubic $\alpha_{21}Y_0^3+\alpha_{02}Y_1^3+\alpha_{10}Y_2^3=0$. This is a Fermat cubic of genus one with a trivial point $[1:1:1]$. By Lemma \ref{lem:fermat_to_weierstrass}, it is an elliptic curve. Again, for many choices of $\ba_2$, this curve has positive rank.
\end{proof}

\begin{thm}
	\label{thm:trivial_points_tower}
	Let $k$ be a number field. For $(s,n) \notin \{(2,2), (2,3), (3,2)\}$, the fiber curves $\Xc_{\ba_n}$ have genus $\geq 2$. For any fixed $\ba_n$, the tower of curves $\{\Xc_{\ba_m}\}_{m \ge n}$ has infinite gonality. Consequently, for $m$ sufficiently large, $\Xc_{\ba_m}(k)$ consists only of trivial points.
\end{thm}
\begin{proof}
	From Proposition \ref{p1}, for $(s,n)$ outside the specified set, $g(\Xc_{\ba_n}) \ge 2$. Furthermore, the gonality $\gamma_k(\Xc_{\ba_m}) \ge (s-1)s^{m-2}$, which tends to infinity as $m \to \infty$. The tower thus has infinite gonality. The conclusion then follows directly from Theorem \ref{xarles_thm} with $d=1$.
\end{proof}

\section{Finiteness and Uniformity Results}
\label{prof-finite}

We now prove Theorems \ref{main1} and \ref{main2}. The key insight is the correspondence established in Proposition \ref{pr1} between curves in $\Cc_{\ba_n}(k)$ and non-trivial rational points on the fiber curves $\Xc_{\ba_n}$.
\[ \Cc_{\ba_n}(k) \longleftrightarrow \Xc_{\ba_n}(k) \setminus \{\text{trivial points}\}. \]

\begin{proof}[Proof of Theorem \ref{main1}]
	Let $n_0 = 4$ if $s=2$ and $n_0=3$ otherwise.
	For $n < n_0$, Theorem \ref{thm:low_genus_cases} shows that the curve $\Xc_{\ba_n}$ is of genus 0 or 1 and can have infinitely many $k$-rational points. This implies that $\Cc_{\ba_n}(k)$ can be infinite.
	
	For $n \geq n_0$, Proposition \ref{p1} shows that the genus $g(\Xc_{\ba_n}) \geq 2$. By Faltings' Theorem \ref{faltings}, the set of rational points $\Xc_{\ba_n}(k)$ is finite. Therefore, the set $\Cc_{\ba_n}(k)$, corresponding to the non-trivial points, must also be finite. Since any curve in $\Cc_{\ba_n}(k)$ has at least $n+1$ rational points whose images in the Jacobian are linearly independent (by the open condition in Proposition \ref{linear-independence-open}), the rank is bounded by a function of the number of points on these finitely many curves.
\end{proof}

\begin{proof}[Proof of Theorem \ref{main2}]
	Assume the weak version of Lang's conjecture holds. Let $n \ge n_0$. The curves $\Xc_{\ba_n}$ form a family of curves of fixed genus $g = g(n,s)$ parameterized by the base $\Uc_{r,n}$. By the Uniformity Theorem \ref{UB}, there exists a bound $N(k,g)$ such that $\#\Xc_{\ba_n}(k) \leq N(k,g)$ for all $\ba_n \in \Uc_{r,n}(k)$. This immediately implies the uniform bound $M_0 = N(k,g)$ for $\#\Cc_{\ba_n}(k)$.
	
	For the second part, assume $g_{r,s}(C) \geq 2$. The family of curves $\Cc \to \Pp^1_{(a,b)}$ has a generic fiber of genus $\geq 2$. By the Correlation Theorem \ref{corel}, for $m$ sufficiently large, the $m$-th fiber product $\Cc^m_{\Pp^1}$ admits a dominant map to a variety of general type $W$. A curve $C \in \Cc_{\ba_m}(k)$ gives rise to a rational point on $\Cc^m_{\Pp^1}$, and thus to a rational point on $W$. By the weak Lang conjecture, the rational points on $W$ are not Zariski dense. The uniformity theorems imply that for $m_0$ large enough, there can be no curves with $m_0+1$ points, so $\Cc_{\ba_m}(k)=\emptyset$ for $m \ge m_0$. A similar argument using Theorem \ref{UGB} proves the strong version, which implies the existence of a bound independent of $k$.
\end{proof}

\section{Proof of the Main Theorem}
\label{proof-main-thm}
We now have all the necessary components to prove Theorem \ref{main}. The proof proceeds by contradiction, leveraging the uniformity results from the previous section and the connection between rank and the number of rational points provided by Theorem \ref{dgh}.

\begin{proof}[Proof of Theorem \ref{main}]
	Let $k$ be a number field of degree $d=[k:\Q]$ over $\Q$. Let $\mathcal{F}$ be the family of smooth curves over $k$ defined by equations of the form $y^s=ax^r+b$ with genus $g_{r,s}(C) = g \geq 1$. We assume, for the sake of contradiction, that the Mordell-Weil rank is unbounded on this family.
	
	This means there exists an infinite sequence of non-isomorphic curves $C_j \in \mathcal{F}$ such that their ranks $r_j = r(J_{C_j}(k))$ tend to infinity as $j \to \infty$.
	
	By Theorem \ref{dgh}, there exists a constant $c = c(g, d)$ such that for each curve $C_j$:
	\[ \#C_j(k) \leq c^{1+r_j}. \]
	Since $r_j \to \infty$, it follows that $\#C_j(k) \to \infty$.
	
	Now, we invoke the strong version of our uniformity result, which follows from the strong Lang conjecture (as in the proof of Theorem \ref{main2} and Theorem \ref{UGB}). This provides an integer $m_1$, depending only on $r$ and $s$, such that for any $m \geq m_1$, the set $\Cc_{\ba_m}(K)$ is empty for any number field $K$. This means no curve in the family $\mathcal{F}$ over any number field $K$ can have $m+1$ rational points with distinct $x$-coordinates if $m \geq m_1$.
	
	Since $\#C_j(k) \to \infty$, we can choose an index $J$ large enough such that $\#C_J(k) > s \cdot (m_1+1)$. A point on $C_J$ is of the form $(\alpha, \beta)$, and for any fixed $x$-coordinate $\alpha$, there are at most $s$ corresponding $y$-coordinates $\beta$. Therefore, the number of distinct $x$-coordinates among the points in $C_J(k)$ is at least $\lceil \#C_J(k)/s \rceil > m_1+1$.
	
	We can thus select $m_1+2$ points from $C_J(k)$ that have distinct $x$-coordinates. Let these points be $(\alpha_0, \beta_0), \dots, (\alpha_{m_1+1}, \beta_{m_1+1})$. Let $\ba_{m_1+1}$ be the point in parameter space corresponding to these $x$-coordinates. By definition, the curve $C_J$ is an element of the set $\Cc_{\ba_{m_1+1}}(k)$.
	
	However, since $m_1+1 > m_1$, our uniformity result implies that $\Cc_{\ba_{m_1+1}}(k)$ must be empty. This is a contradiction, as we have found a curve $C_J$ that lies in this supposedly empty set.
	
	Therefore, our initial assumption must be false. The rank of Jacobian varieties in the family $\mathcal{F}$ must be bounded.
\end{proof}

\section{A Rank 17 Elliptic Curve over $\Q$}
\label{examples}
We illustrate the abstract construction with a concrete example to show how specific high rank elliptic curve  with a known set of rational points correspond to a single rational point on the high-genus fiber variety $\Xc_{\ba_n}$.

In 2016, Elkies discovered an elliptic curve of rank 17 with $j=0$, given by
\[ E: y^2 = x^3 + b_0, \quad b_0 = 24537619889008718205152851658505801. \]
This curve contains 17 known independent points $P_0, \dots, P_{16} \in E(\Q)$, as follows:	
{\small
	\begin{align*}
		P_0   & =  [-249954149276, 94452185380426435],    & P_1  & =  [-218829008658, 118569576333381183],           \\
		P_2   & =  [-110315760690, 152299457785937151],   & P_3  & =  [-12083686365, 156639252691623474],            \\
		P_4   & =  [179588218407, 174154202398188288],    & P_5  & =  [194693247690, 178654854781822599],            \\
		P_6   & =  [481938369495, 369425010854453724],    & P_7  & =  [527526224524, 413931980240076925],           \\
		P_8   & =  [532637728899, 419104420151289750],    & P_9  & =  [660796972800, 559532270810391651],           \\
		P_{10} & =  [891937317975, 856808203106532276],   & P_{11} & =  [1369152212199, 1609695603071293320],     \\
		P_{12} & =  [1556910033324, 1948958451538253955], & P_{13} & =  [2095375244992, 3037184017947911267],    \\
		P_{14} & =  [3020920353232, 5252935870900542563], & P_{15} & =  [45908680009155, 311058636438867847974],  \\ 
		P_{16} & =  [209109621212430, 3023855428577131273599].   
\end{align*}}
Let $\alpha_i = x(P_i)$ be their $x$-coordinates. We consider the point $\ba_{16} = (\alpha_0, \dots, \alpha_{16})$. This single elliptic curve corresponds to a single rational point
$[y(P_0): \cdots : y(P_{16})]$ 
on the fiber curve $\Xc_{\ba_{16}}$ given by
the following equations,

{\scriptsize
	\begin{align*}  
		c Y_2^2  & = 14273909518752011104805996875187576 Y_1^2 - 9136380410012945144031390143517312 Y_0^2,  \\
		c Y_3^2  & = 15614640160651830564600703341019451 Y_1^2 -10477111051912764603826096609349187 Y_0^2, \\
		c Y_4^2  & = 21408470889810690063581042253561719 Y_1^2 -16270941781071624102806435521891455 Y_0^2, \\
		c Y_5^2  & = 22996341813975679987992445860305576 Y_1^2 -17858812705236614027217839128635312 Y_0^2,\\
		c Y_6^2  & = 127553623321674810616820424010658951 Y_1^2 -122416094212935744656045817278988687 Y_0^2,\\
		c Y_7^2  & = 162418468942332992713014707470646400 Y_1^2 -157280939833593926752240100738976136 Y_0^2,\\
		c Y_8^2  & = 166727299667210364694533366008253275 Y_1^2 -161589770558471298733758759276583011 Y_0^2, \\
		c Y_9^2  & = 304155146755095019623144945563696576 Y_1^2 -299017617646355953662370338832026312 Y_0^2, \\
		c Y_{10}^2 & = 725199081587506223753723959382930951 Y_1^2 -720061552478767157792949352651260687 Y_0^2,\\
		c Y_{11}^2 & = 2582198719223916255279881435029813175 Y_1^2  -2577061190115177189319106828298142911 Y_0^2 , \\
		c Y_{12}^2 & = 3789517830499250148872819507626332800 Y_1^2 -3784380301390511082912044900894662536 Y_0^2,  \\
		c Y_{13}^2 & = 9215565543555079748052029625658736064 Y_1^2 -9210428014446340682091255018927065800 Y_0^2,\\
		c Y_{14}^2 & = 27584414048470503122920101305327799744 Y_1^2  -27579276519361764056959326698596129480 Y_0^2,\\
		c Y_{15}^2 & = 96757466381992441404259415168107529095451 Y_1^2 -96757461244463332665193454393500797425187 Y_0^2,\\
		c Y_{16}^2 & = 9143701644014170929876417817964781347603576 Y_1^2 -9143701638876641821137351857190174615933312 Y_0^2,
\end{align*}}
where $c=513752910873906596077460673167026$.

According to Proposition \ref{p1}, the genus of this  curve is:
\[ g(\Xc_{\ba_{16}}) = 1 + \frac{2^{15}((16-1)(2-1)-2)}{2} =   212993. \]
The existence of this famous rank 17 curve is thus encoded as the existence of a single rational point on a curve of genus 212,993. 
%The explicit equations defining this curve are given in the appendix. 
This dramatically illustrates the principle of the construction: information about a set of points on one curve is transformed into information about a single point on a much more complex curve.

\bibliography{BMW-bib-1}

\end{document}